\documentclass[a4paper]{amsart}

\usepackage[left=2cm,right=2cm,top=3.5cm,bottom=3cm]{geometry}

\usepackage{amsmath,amssymb,amscd,amsfonts}
\usepackage{mathrsfs}
\usepackage{latexsym}
\usepackage{graphicx}
\usepackage[all]{xy}
\usepackage{eucal}
\usepackage{verbatim}
\usepackage{amsaddr}

\allowdisplaybreaks

\newcommand{\N}{\mathbb{N}}
\newcommand{\Z}{\mathbb{Z}}
\newcommand{\Q}{\mathbb{Q}}
\newcommand{\F}{\mathbb{F}} % finite field and/or residue field
\newcommand{\M}{\mathfrak{M}}
\renewcommand{\S}{\mathcal{S}}
\newcommand{\Sel}{\mathfrak{S}}
\newcommand{\sri}{\twoheadrightarrow}
\newcommand{\iri}{\hookrightarrow}
\renewcommand{\l}{\ell}

\renewcommand{\L}{\Lambda}
\newcommand{\G}{\Gamma}
\newcommand{\g}{\gamma}
\newcommand{\dl}[1]{\lim_{\buildrel \longrightarrow\over{#1}}}
\newcommand{\il}[1]{\lim_{\buildrel \longleftarrow\over{#1}}}

\renewcommand{\leq}{\leqslant}
\renewcommand{\geq}{\geqslant}

\newtheorem{thm}{Theorem}[section]
\newtheorem{prop}[thm]{Proposition}
\newtheorem{lem}[thm]{Lemma}
\newtheorem{cor}[thm]{Corollary}
\newtheorem{defin}[thm]{Definition}
\newtheorem{rem}[thm]{Remark}

\input cyracc.def
\font\tencyr=wncyr10
scaled \magstephalf
\def\cyr{\tencyr\cyracc}
\newcommand{\ts}{\mbox{\cyr Sh}}

\DeclareMathOperator{\Coind}{Coind}
\DeclareMathOperator{\Ind}{Ind}
\DeclareMathOperator{\Hom}{Hom}
\DeclareMathOperator{\Gal}{Gal}
\DeclareMathOperator{\cd}{cd}

\DeclareMathOperator{\ch}{\mathfrak{ch}}

\title[Euler characteristic and Akashi series]{Euler characteristic and Akashi series for Selmer groups over global function fields}

\author{Andrea Bandini}

\address{ {\footnotesize Universit\`a degli Studi di Parma \\
Dipartimento di Matematica\\
Parco Area delle Scienze, 53/A \\
43124 Parma - Italy\\
andrea.bandini@unipr.it
}}

\author{Maria Valentino} \thanks{M. Valentino is supported by an outgoing Marie-Curie fellowship of INdAM}

\address{{\footnotesize King's College London\\
Department of Mathematics\\
Strand, London WC2R 2LS\\
United Kingdom\\
maria.valentino@kcl.ac.uk }}



\begin{document}

\begin{abstract}
Let $A$ be an abelian variety defined over a global function field $F$ of positive characteristic $p$ and let
$K/F$ be a $p$-adic Lie extension with Galois group $G$. We provide a formula for the Euler characteristic
$\chi(G,Sel_A(K)_p)$ of the $p$-part of the Selmer group of $A$ over $K$. In the special case $G=\Z_p^d$
and $A$ a constant ordinary variety, using Akashi series, we show how the Euler characteristic of the dual
of $Sel_A(K)_p$ is related to special values of a $p$-adic $\mathcal{L}$-function.
\end{abstract}

\maketitle

\section{Introduction}
Let $p\in \Z$ be  a prime and let $G$ be a profinite $p$-adic Lie group of finite dimension $d\geq 1$ and without elements of order $p$.
Let $M$ be a $G$-module and consider the following properties\begin{itemize}
\item[{\bf 1.}] $H^i(G,M)$ is finite for any $i\geqslant 0$;
\item[{\bf 2.}] $H^i(G,M)=0$ for all but finitely many $i$.
\end{itemize}

\begin{defin}\label{EuCharDef}
If a $G$-module $M$ verifies {\bf 1} and {\bf 2}, the {\em Euler characteristic} of $M$ is defined as
\[ \chi(G,M):=\prod_{i\geq 0} |H^i(G,M)|^{(-1)^i} \,. \]
We will use the following notation for the {\em $n$-th Euler characteristic}
\[ \chi^{(n)}(G,M):=\prod_{i\geq n} |H^i(G,M)|^{(-1)^i} \,.\]
\end{defin}

\noindent Denote by $\L(G)$ the Iwasawa algebra associated to $G$. Let $\M_H(G)$ be the category of
all finitely generated $\L(G)$-modules which are finitely generated also as
$\L(H)$-modules, where $H$ is a closed normal subgroup of $G$ such that $\G:=G/H\simeq \Z_p$.
If $h_1$ and $h_2$ are any two non-zero elements of $Q(\L(\G))$ (the field of fraction of the Iwasawa algebra $\L(\G)\,$),
we write $h_1 \sim h_2$ if $h_1h_2^{-1}$ is a unit in $\L(\G)$. \\
In \cite[Section 4]{CSS}, the authors attach to any $M$ in $\M_H(G)$ a non-zero element $f_{M,H}$ of $Q(\L(\G))$
which is defined as follows. For any $M\in \M_H(G)$, the homology groups $H_i(H,M)$ are finitely generated torsion
$\L(\G)$-modules for all $i\geqslant 0$ (see \cite[Lemma 3.1]{CFKSV}). We denote by
$\ch_{\L(\G)}(H_i(H,M))$ their characteristic elements.

\begin{defin}
With notations as above, the {\em Akashi series} of $M\in \M_H(G)$ is
\[  f_{M,H}:= \prod_{i\geqslant 0} \ch_{\L(\G)}(H_i(H,M))^{(-1)^i}\,. \]
\end{defin}

\noindent The product in the above definition is finite because $H_i(H,M) = 0$ for $i > d$, and it is well
defined up to $\sim$, because each $\ch_{\L(\G)}(H_i(H,M))$ is well defined up to multiplication by a unit
in $\L(\G)$.

Let $A/F$ be an abelian variety defined over a global function field $F$ of characteristic $p$. For any field extension
$L/F$ we denote by $A(L)[p^\infty]$ the $p$-torsion of $A$ defined over $L$ and
by $Sel_A(L)_p$ the $p$-part of the Selmer group of $A$ over $L$.
When $G$ is the Galois group of a field extension $L/F$, the study of the Euler characteristic $\chi(G,Sel_A(L)_p)$ and of the
Akashi series associated to these data, is a first step towards understanding the relation, predicted by the Iwasawa Main Conjecture,
between a characteristic element for (the Pontrjagin dual of) $Sel_A(L)_p$ and a suitable $p$-adic $L$-function.
The aim of this paper is to provide formulas for the Euler characteristic (actually the {\em truncated Euler characteristic}
$\frac{\chi}{\chi^{(2)}}$ in the terminology of \cite[Section 3]{CSS}) of Selmer groups and for the Akashi series
of the Pontrjagin dual of the Selmer group in the function field case (the number field case has been extensively
studied, see, for example, \cite{VO}, \cite{Z1}, \cite{Z3} and the references there). In particular we prove the following
(see Section \ref{DescDiag} and Theorem \ref{ThmEuChp} for a precise definition of all the terms involved)

\begin{thm}\label{IntroThm1}
Let $G=\Gal(K/F)$ be as above and assume that\begin{itemize}
\item[{\it i)}] $K/F$ is unramified outside a finite set of places $S$ which also contains all primes of bad reduction for $A$;
\item[{\it ii)}] $Sel_A(F)_p$ and $Sel_{A^t}(F)_p$ are finite ($A^t$ is the dual abelian variety of $A$);
\item[{\it iii)}] $\chi(G,A(K)[p^\infty])$ is well defined;
\item[{\it iv)}] $H^i(G_v,A(K_w))[p^\infty]$ (for any $w$ dividing a ramified place $v\in S$) is finite for $i=1,2$.
\end{itemize}
Then,
\[ \begin{array}{ll} \displaystyle{\frac{\chi(G,Sel_A(K)_p)}{\chi^{(2)}(G,Sel_A(K)_p)}} & =
\displaystyle{\frac{\chi(G,A(K)[p^\infty])}{\chi^{(3)}(G,A(K)[p^\infty])}
\cdot\frac{|\ts(A/F)[p^\infty]||NS(\psi_K)|}{|A^t(F)[p^\infty]||A(F)[p^\infty]|} }
\cdot \prod_{\begin{subarray}{c} v\in S \\v\ inert \end{subarray}}p^{\nu_v}\\
\ &  \cdot\displaystyle{\prod_{\begin{subarray}{c} v\in S \\v\ ram. \end{subarray}}
\frac{|H^1(G_v,A(K_w))[p^\infty]|}{|H^2(G_v,A(K_w))[p^\infty]|} } \,. \end{array} \]
\end{thm}

This formula does not involve (at least directly) special values of $p$-adic or global $L$-functions, which
are one of the main features of the classical formulas for number fields (and also for function fields
when one considers the $\ell$-part of the Selmer group for some prime $\ell\neq p$, see \cite{V}).
One of the main problems is the (still ongoing) search for the right analogue of a $p$-adic $L$-function
in our setting. One relevant exception is provided by \cite{LLTT} for the case of constant abelian varieties.
Their $p$-adic $\mathcal{L}$-function will appear in a special case of the formula for the Akashi series
(see Corollary \ref{AkSerStickEl}).

\subsection{Setting and notations}
Before moving on we briefly describe the setting in which we shall work.\\
Let $F$ be a global function field of characteristic $p>0$. Consider a $p$-adic Lie extension  $K/F$ which is unramified outside a finite
set of places and let $G$ denote its Galois group. We assume that $G$ has finite dimension $d$ (as $p$-adic Lie group)
and no elements of order $p$. Under these hypotheses $G$ has finite (Galois) cohomological dimension,
which is equal to its dimension as $p$-adic Lie group (\cite[Corollaire (1) p. 413]{Se2}).\\
Consider an abelian variety $A/F$ and let $S$ be a finite set of places of $F$ containing the
primes of bad reduction for $A$ and those which ramify in $K/F$. We denote by $A^t$ its dual
abelian variety and, as usual, $A[p^n]$ will be the scheme of $p^n$-torsion points of $A$, with
$A[p^\infty]=\displaystyle{\dl{n} A[p^n]}$.

For any field $L$, we denote by $\Sigma_L$ the set of places of $L$ and, for any $v\in\Sigma_L$, we let $L_v$ be
the completion of $L$ at $v$.\\
We define Selmer groups via the usual cohomological techniques and, since we deal mainly with
the flat scheme of torsion points, we shall use the flat cohomology groups $H^i_{fl}$ (since $A$ and $p$
are fixed throughout the paper, except for a few appearences of $A^t$, we usually forget about them in the notation
for the Selmer group).

\begin{defin}
For any finite extension $L/F$, the {\em $p$-part of the Selmer group} of $A$ over $L$ is
\[ Sel(L):=Sel_A(L)_p:= Ker\left\{ H^1_{fl}(X_L,A[p^\infty]) \rightarrow
\prod_{w\in \Sigma_L} H^1_{fl}(X_{L_w},A[p^\infty])/Im(\kappa_w)\right\}\]
where $X_L:=Spec(L)$ and $\kappa_w$ is the usual (local) Kummer map. For infinite extensions we define the Selmer groups by
taking direct limits on the finite subextensions.
\end{defin}

\noindent Letting $L$ vary through subextensions of $K/F$, the groups $Sel_A(L)_\l$ admit natural actions
by $\mathbb{Z}_p$ and $G$. Hence they are modules over the Iwasawa algebra $\L(G)=\Z_p[[G]] := \displaystyle{\il{U} \Z_p [G/U]}$,
where the limit is taken on the open normal subgroups of $G$.

Let $F_S$ be the maximal separable pro-$p$-extension of $F$ unramified outside $S$,
so that $K\subseteq F_S\,$. We recall that the $p$-cohomological (Galois) dimension of $G_S(F):=\Gal(F_S/F)$ is
$\cd_p(G_S(F))=1$ (\cite[Theorem 10.1.12 (iv)]{NSW}). For any field extension $L/F$ contained in $F_S$, we denote by
$G_S(L)$ the Galois group $\Gal(F_S/L)$.

Let $H$ be a closed subgroup of $G$. For every $\L(H)$-module
$N$ we consider the $\L(G)$-modules (some texts, e.g. \cite{NSW}, switch the definitions of $\Ind^H_G(N)$
and $\Coind^H_G(N)$)
\[ \Coind^H_G(N):={\rm Map}_{\L(H)}(\L(G), N)\quad {\rm and}\quad
\Ind_H^G(N):=\L(G)\otimes_{\L(H)} N\,. \]

For any $v\in\Sigma_F$, fix a place $w\in\Sigma_K$ lying above $v$ and let $G_v:=\Gal(K_w/F_v)$ be the associated
decomposition group. We use the isomorphism
\[ H_{fl}^1(X_{L_w},A[p^\infty])/Im(\kappa_w) \simeq H^1_{fl}(X_{L_w},A)[p^\infty] \]
(coming from the Kummer sequence) and the more convenient notation of $\Coind$ modules, to get
our working definition for $Sel(K)$, i.e.,
\begin{equation}\label{EqDefSel}
Sel(K):=Sel_A(K)_p = Ker\left\{ H^1_{fl}(X_K,A[p^\infty]) {\buildrel \psi_K\over{-\!\!\!-\!\!\!\longrightarrow}}
\prod_{v\in \Sigma_F} \Coind^{G_v}_G H^1_{fl}(X_{K_w},A)[p^\infty] \right\} \,.
\end{equation}

If $L/F$ is a finite extension the group $Sel(L)$ is a cofinitely generated $\Z_p$-module
(see, e.g. \cite[III.8 and III.9]{Mi1}). One defines the {\em Tate-Shafarevich group} $\ts(A/L)$
as the group that fits into the exact sequence
\[ A(L)\otimes \Q_p / \Z_p \iri Sel(L) \sri \ts(A/L)[p^\infty]\ .\]
Whenever we assume that $Sel(L)$ is finite (we shall mainly use this hypothesis with $L=F$), we have that the
$\Z$-rank of $A(L)$ is 0, hence
\begin{equation}\label{EqSha}
A(L)\otimes \Q_p/\Z_p = 0\quad \textrm{and}\quad |Sel(L)|= |\ts(A/L)[p^\infty]|\ .
\end{equation}

\noindent For a $\L(G)$-module $M$, we denote by $M^{\vee}:=\Hom_{cont}(M,\mathbb{C}^{\ast})$ its Pontrjagin dual.
In the cases considered in this paper, $M$ will be a (mostly discrete) topological $\Z_p$-module,
so that $M^{\vee}$ can be identified with $\Hom_{cont}\!(M,\Q_p/\Z_p)$
and it has a natural structure of $\Z_p$-module.

\section{Euler characteristic}\label{leqp}
Before moving to the proof of the Euler characteristic formula we list some intermediate results
which will be useful for the computation.

\subsection{Cohomological lemmas}\label{CohoLemp}
Here we are going to collect some results on flat (local and global) cohomology groups.

\begin{lem}\label{Hflzero}
Let $L/F$ be any field extension contained in $F_S$ and let $w$ be any prime of $L$ lying over $v\in \Sigma_F$.
Then
\begin{itemize}
\item[{\bf 1.}] {$H^i_{fl}(X_L,A[p^\infty])=0$ $\forall\, i\geq 2;$}
\item[{\bf 2.}] {$H^i_{fl}(X_{L_w},A[p^\infty])=0$ $\forall\, i\geq 2.$}
\end{itemize}
\end{lem}

\begin{proof} {\bf 1.}
The map $X_{F_S}\to X_L$ is a Galois covering with Galois group $G_S(L)$. Then we have the Hochschild-Serre spectral sequence:
\[  E_2^{n,m}(F_S/L):=H^n(G_S(L),H^m_{fl}(X_{F_S},A[p^\infty]))\Rightarrow H^{n+m}_{fl}(X_L ,A[p^\infty]).  \]
If $L/F$ is finite, $\cd_p(G_S(L))=1$ because $L$ is still a global field (see, \cite[Theorem 10.1.12 (iv)]{NSW}).
When $L/F$ is not finite, $\cd_p(G_S(L))\leq 1$ since $G_S(L)$ is closed in $G_S(F)$. Anyway, we have that
\[ H^n(G_S(L),H^m_{fl}(X_{F_S},A[p^\infty]))=0\ \ \forall\,n>1\ \mathrm{and}\ m\in\N \, .\]
Thanks to {\cite[Lemma 2.1.4]{NSW}} we have
\begin{equation*} H^n(G_S(L),H^m_{fl}(X_{F_S},A[p^\infty]))\simeq H^{n+m}_{fl}(X_L,A[p^\infty])\ \
\forall\,n\geq 1\ \mathrm{and}\ \mathrm{every}\ m\ . \end{equation*}
In particular,
\begin{align*} H^{i}_{fl}(X_L,A[p^\infty]) & \simeq H^i(G_S(L),H^0_{fl}(X_{F_S},A[p^\infty]))\\
 & \simeq H^i(G_S(L),A(F_S)[p^\infty])=0  \ \ \forall\,i\geq 2. \end{align*}
{\bf 2.} The proof works as in {\bf 1} (recalling that the decomposition group of $v$ in $G_S(F)$ is a closed subgroup,
so it has cohomological dimension $\leq 1$ as well).
\end{proof}

\begin{lem}\label{HGzero}
Let $K/F$ be any $p$-adic Lie extension contained in $F_S$ with Galois group $G$.
For any fixed place $w$ of $K$ dividing $v\in\Sigma_F$, let $G_v$ be the corresponding decomposition group. Then,
\begin{itemize}
\item[{\bf 1.}] {$H^i(G, H^1_{fl}(X_K,A[p^\infty]))=0\ \ \forall\,i\geq1\,;$}
\item[{\bf 2.}] {$H^i(G_v,H^1_{fl}(X_{K_w},A[p^\infty]))=0\ \ \forall\,i\geq 1\,.$}
\end{itemize}
\end{lem}

\begin{proof} {\bf 1.} The Galois covering $X_K \to X_F$ with Galois group $G$ gives us the Hochschild-Serre
spectral sequence
\[ E_2^{n,m}(K/F):=H^n(G,H^m_{fl}(X_{K},A[p^\infty]))\Rightarrow H^{n+m}_{fl}(X_F ,A[p^\infty])\]
with $H^n(G, H^m_{fl}(X_K, A[p^\infty]))=0\ \ \forall\ n\ \mathrm{and}\  m>1$ (by the previous lemma).
By \cite[Lemma 2.1.4]{NSW} we have that
\[ H^n(G, H^m_{fl}(X_K, A[p^\infty]))\simeq H^{n+m}_{fl}(X_F,A[p^\infty])\ \ \forall\,n\ \mathrm{and}\ \forall\,m\geq 1\,. \]
Thanks to the previous lemma
\[ H^i(G, H^1_{fl}(X_K, A[p^\infty]))\simeq H^{i+1}_{fl}(X_F,A[p^\infty])=0 \ \ \forall\,i\geq 1\,. \]
{\bf 2.} The argument is the same of part {\bf 1}.
\end{proof}

\begin{prop}\label{HiSel0}
For any $i\geqslant 2$ we have
\[ H^i(G, Sel(K)) \simeq H^{i-1}(G,Im(\psi_K)) \simeq \displaystyle{\prod_{v\in \Sigma_F} H^{i-1}(G_v, Im(\psi_{K,v}))} \]
(where $\psi_K$ is defined in \eqref{EqDefSel} and $Im(\psi_{K,v})$ is the image of $\psi_K$ in $H^1_{fl}(X_{K_w},A)[p^\infty]$
for some fixed $w$ lying above $v$).
Moreover, if $i\geqslant 3$, then
\[ H^i(G, Sel(K)) \simeq H^{i-1}(G,Im(\psi_K)) \simeq \displaystyle{\prod_{\begin{subarray}{c} v\in S \\v\ ram. \end{subarray}}
H^{i-1}(G_v, Im(\psi_{K,v}))}\,.\]
\end{prop}

\begin{proof}
Consider the sequence
\[ Sel(K) \iri H^1_{fl}(X_K,A[p^\infty]) \stackrel{\psi_K}{-\!\!\!-\!\!\!-\!\!\!-\!\!\!\sri}
Im(\psi_K)=\prod_{v\in \Sigma_F}\Coind^{G_v}_G Im(\psi_{K,v})\,. \]
Taking its cohomology with respect to $G$ we obtain
\[ \begin{xy}
(-45,0)*+{Sel(K)^G}="v1";(0,0)*+{H^1_{fl}(X_K,A[p^\infty])^G}="v2";(55,0)*+
{Im(\psi_K)^G\to}="v3";
(-45,-6)*+{\ldots}; (0,-6)*+{\ldots};  (50,-6)*+{\ldots};
(-45,-13)*+{H^i(G,Sel(K))}="v4";(0,-13)*+{H^i(G,H^1_{fl}(X_K,A[p^\infty]))}="v5";(55,-13)*+
{H^i(G,Im(\psi_K))\to}="v6";
(-45,-19)*+{\ldots}; (0,-19)*+{\ldots};  (50,-19)*+{\ldots};
(-45,-26)*+{H^d(G,Sel(K))}="v7";(0,-26)*+{H^d(G,H^1_{fl}(X_K,A[p^\infty]))}="v8";(55,-26)*+
{H^d(G,Im(\psi_K))}="v9";
{\ar@{^{(}->} "v1";"v2"};{\ar@{->} "v2";"v3"};{\ar@{->} "v4";"v5"};{\ar@{->} "v5";"v6"};
{\ar@{->} "v7";"v8"};{\ar@{->>} "v8";"v9"};
\end{xy}\]
(recall that, by assumption, $G$ has no elements of order $p$ and that it has finite cohomological
dimension $d$).\\
Thanks to Lemmas \ref{Hflzero} and \ref{HGzero} part {\bf 1}, the above sequence provides another sequence
\begin{equation}\label{EqHiSel}
Sel(K)^G \iri H^1_{fl}(X_K,A[p^\infty])^G \rightarrow Im(\psi_K)^G \sri H^1(G,Sel(K))
\end{equation}
and isomorphisms
\[ H^i(G, Sel(K))\simeq H^{i-1}(G,Im(\psi_K)) \simeq \displaystyle{\prod_{v\in \Sigma_F} H^{i-1}(G_v, Im(\psi_{K,v}))}
\quad \forall\ i\geqslant 2 \]
(the last isomorpshism follows from $Im(\psi_K)=\displaystyle{\prod_{v\in \Sigma_F}}\Coind^{G_v}_G Im(\psi_{K,v})$ and Shapiro's Lemma).\\
When $v$ is unramified, we have that $G_v$ is 0 or $\Z_p$, hence of cohomological dimension $\leqslant 1$. Therefore
for all those primes
\[ H^{i-1}(G_v, Im(\psi_{K,v}))=0 \quad \forall\ i\geqslant 3 \,.\]
\end{proof}

\begin{cor}\label{SurjpsiK}
If $\psi_K$ is surjective, then for any $i\geqslant 2$ we have
\[ \begin{array}{ll} H^i(G, Sel(K)) & \simeq
\displaystyle{\prod_{v\in \Sigma_F} H^{i-1}(G_v, H^1_{fl}(X_{K_w}A)[p^\infty])} \\
\ & \simeq
\displaystyle{\prod_{\begin{subarray}{c} v\in S \\v\ ram. \end{subarray}} H^i(G_v,A(K_w)\otimes \Q_p/\Z_p)\,.}\end{array} \]
Moreover (whenever all terms are defined)
\[ \chi(G,Sel(K))=\frac{|H^1_{fl}(X_K,A[p^\infty])^G|}{\displaystyle{\prod_{v\in\Sigma_F}}|H^1_{fl}(X_{K_w},A[p^\infty])^{G_v}|}\cdot
\prod_{v\in \Sigma_F} \chi(G_v,A(K_w)\otimes \Q_p/\Z_p) \ .\]
\end{cor}

\begin{proof}
For the first isomorphism just substitute $Im(\psi_K)$ with
$\displaystyle{\prod_{v\in\Sigma_F}}\Coind^{G_v}_G H^1_{fl} (X_{K_w},A)[p^\infty]$ in the previous proposition.
For the second one, use the cohomology sequence of
\begin{equation}\label{ExSeqAKw}
 A(K_w)\otimes \Q_p/\Z_p \iri H^1_{fl}(X_{K_w},A[p^\infty]) \sri H^1_{fl}(X_{K_w},A)[p^\infty]
 \end{equation}
together with Lemma \ref{HGzero} part {\bf 2}. The unramified places are eliminated as in the previous proposition
(but note that this now holds for $i=2$ as well).
Therefore
\[ \chi^{(2)}(G,Sel(K)) = \prod_{\begin{subarray}{c} v\in S \\v\ ram. \end{subarray}}\chi^{(2)}(G_v,A(K_w)\otimes \Q_p/\Z_p) \]
and equation \eqref{EqHiSel} shows that
\[ \frac{\chi(G,Sel(K))}{\chi^{(2)}(G,Sel(K))}=
\frac{|H^1_{fl}(X_K,A[p^\infty])^G|}{\displaystyle{\prod_{v\in\Sigma_F}} |H^0(G_v, H^1_{fl}(X_{K_w},A)[p^\infty])|}\,.\]
From the cohomology sequence of \eqref{ExSeqAKw} and Lemma \ref{HGzero} part {\bf 2} one gets
\[ \begin{array}{ll} |H^0(G_v, H^1_{fl}(X_{K_w},A)[p^\infty])| & =\displaystyle{
\frac{|H^1_{fl}(X_{K_w},A[p^\infty])^{G_v}||H^1(G_v,A(K_w)\otimes \Q_p/\Z_p)|}{|A(F_v)\otimes \Q_p/\Z_p|}} \\
\ & = |H^1_{fl}(X_{K_w},A[p^\infty])^{G_v}|\cdot \displaystyle{
\frac{\chi^{(2)}(G_v,A(K_w)\otimes \Q_p/\Z_p)}{\chi(G_v,A(K_w)\otimes \Q_p/\Z_p)}\,. }
\end{array} \]
Putting everything together (and observing that for unramified primes $\chi^{(2)}(G_v,A(K_w)\otimes \Q_p/\Z_p)=1$)
we get the final formula.
\end{proof}

\subsection{Descent diagrams}\label{DescDiag}
Consider the sequence
\[ Sel(K) \iri H^1_{fl}(X_K,A[p^\infty]) \stackrel{\psi_K}{-\!\!\!-\!\!\!-\!\!\!-\!\!\!\sri} Im(\psi_K) \]
and let $\psi_K^G : H^1_{fl}(X_K,A[p^\infty])^G \rightarrow Im(\psi_K)^G$ be the induced map in cohomology.
Then equation \eqref{EqHiSel} shows that
\[ H^1(G,Sel(K))\simeq Coker(\psi_K^G) \,.\]
We consider the classical descent diagram, which has been already used to study the structure of Selmer groups
as modules over some Iwasawa algebra (see, e.g., \cite{BV1} and the references there)
\begin{equation}\label{Diag1p}
\xymatrix{ Sel(F) \ar[d]^\alpha \ar@{^(->}[r] & H^1_{fl}(X_F,A[p^\infty]) \ar@{->>}[r]^{\qquad\psi_F} \ar[d]^\beta &
Im(\psi_F) \ar[d]^{\beta'} \\
Sel(K)^G \ar@{^(->}[r] & H^1_{fl}(X_K,A[p^\infty])^G \ar@{->>}[r]^{\qquad\psi_K^G} &
Im(\psi_K^G)  \ ,}
\end{equation}
and also
\begin{equation}\label{Diag2p}
\xymatrix { Im(\psi_F) \ar[d]^{\beta'} \ar@{^(->}[r] &
\displaystyle{\prod_{v\in \Sigma_F} H^1_{fl}(X_{F_v},A)[p^\infty]} \ar@{->>}[r] \ar[d]^\eta &
Coker(\psi_F) \ar[d]^{\eta'} \\
Im(\psi_K^G) \ar@{^(->}[r] & Im(\psi_K)^G \ar@{->>}[r] & H^1(G,Sel(K))  \ .}
\end{equation}
Moreover we have a natural inclusion $ Im(\psi_K)^G \iri \displaystyle{\prod_{v\in \Sigma_F} \Coind^{G_v}_G H^1_{fl}(X_{K_w},A)[p^\infty]^G }$
whose cokernel we denote by $NS(\psi_K)$. Note that if $\psi_K$ is surjective, then $NS(\psi_K)=0$ and, for lack of a better
description, we could say that $NS(\psi_K)$ measures the defect of surjectivity of $\psi_K$.
The commutative diagram
\begin{equation}\label{Diag3p}
\xymatrix { \displaystyle{\prod_{v\in \Sigma_F} H^1_{fl}(X_{F_v},A)[p^\infty] } \ar@{=}[r] \ar[d]^{\eta}&
\displaystyle{\prod_{v\in \Sigma_F} H^1_{fl}(X_{F_v},A)[p^\infty] \ar[d]^{\gamma}} \\
Im(\psi_K)^G \ar@{^(->}[r] & \displaystyle{\prod_{v\in \Sigma_F} \Coind^{G_v}_G H^1_{fl}(X_{K_w},A)[p^\infty]^G } \ar@{->>}[r] &
NS(\psi_K) }
\end{equation}
shows that
\begin{equation} \label{EqDiag3p}
Ker(\gamma)\simeq Ker(\eta)\quad{\rm and}\quad NS(\psi_K)\simeq Coker(\gamma)/Coker(\eta)\,.
\end{equation}
As already done in \cite{V} we can derive a formula for $\chi(G,Sel(K))$ using the
cardinalities of kernels and cokernels appearing in diagrams \eqref{Diag1p}, \eqref{Diag2p} and \eqref{Diag3p} (since our
goal is a formula for the Euler characteristic we assume that it is well defined, i.e., all the relevant modules are finite).

\noindent We first provide a general formula and then give more precise information on each factor.

\noindent From the right vertical sequence of \eqref{Diag2p}, the snake lemma sequences of diagrams \eqref{Diag1p} and \eqref{Diag2p},
and equation \eqref{EqDiag3p} one has
\[ \begin{array}{lcl} |H^1(G,Sel(K))|& = & \displaystyle{|Coker(\psi_F)|\cdot \frac{|Coker(\eta')|}{|Ker(\eta')|}} \\
\ & = & \displaystyle{|Coker(\psi_F)|\cdot \frac{|Ker(\beta')|}{|Coker(\beta')|}\cdot\frac{|Coker(\eta)|}{|Ker(\eta)|}} \\
\ & = & \displaystyle{\frac{|Coker(\psi_F)|}{|NS(\psi_K)|}\cdot
\frac{|Coker(\gamma)|}{|Ker(\gamma)|}\cdot\frac{|Ker(\beta)|}{|Coker(\beta)|}\cdot\frac{|Coker(\alpha)|}{|Ker(\alpha)|}}\ .\end{array} \]

\begin{lem}\label{H^0/H^1}
One has
\[ \frac{\chi(G,Sel(K))}{\chi^{(2)}(G,Sel(K))} =
\frac{|Sel(F)||NS(\psi_K)|}{|Coker(\psi_F)|}\cdot\frac{|Coker(\beta)|}{|Ker(\beta)|}\cdot\frac{|Ker(\gamma)|}{|Coker(\gamma)|} \]
(assuming that all terms are finite).
\end{lem}

\begin{proof} The formula follows from the
previous equation and the fact that (by the right vertical sequence of diagram \eqref{Diag1p})
\[ \frac{|Coker(\alpha)|}{|Ker(\alpha)|} = \frac{|Sel(K)^G|}{|Sel(F)|} \ .\qedhere \]
\end{proof}

\begin{lem}\label{LocHochSer}
In diagram \eqref{Diag1p} one has
\[ Ker(\beta) = H^1(G,A(K)[p^\infty]) \quad and \quad
Coker(\beta) = H^2(G,A(K)[p^\infty]) \ .\]
In diagram \eqref{Diag3p} one has
\begin{equation}\label{EqKerGamma}
Ker(\gamma)\simeq \prod_{v\in \Sigma_F} H^1(G_v,A(K_w))[p^\infty] \quad and \quad
Coker(\gamma)\simeq \prod_{v\in \Sigma_F} H^2(G_v,A(K_w))[p^\infty] \ ,\end{equation}
where $w$ is a fixed place of $K$ dividing $v$. Moreover, $H^1(G_v,A(K_w))=0$ for any $v\not\in S$,
while for any unramified place $v$ we have
$H^2(G_v,A(K_w))[p^\infty]=0$.
\end{lem}

\begin{proof} For the map $\beta$ use the five term exact sequence of the Hochschild-Serre spectral sequence
(see \cite[Proposition III.2.20 and Remark III.2.21]{Mi2}) recalling that $H^2_{fl}(X_F,A[p^\infty])=0$ (by
Lemma \ref{Hflzero})  to get $Ker(\beta) = H^1(G,A(K)[p^\infty])$ and $Coker(\beta) = H^2(G,A(K)[p^\infty])$.
For the map $\g$, by Shapiro's Lemma
\[ \Coind_G^{G_v} H^1_{fl}(X_{K_w},A)[p^\infty]^G \simeq H^1_{fl}(X_{K_w},A)[p^\infty]^{G_v} \ ,\]
(for some fixed place $w$ dividing $v$). The map $\gamma$ can be written as
\[ \gamma\,:\, \prod_{v\in \Sigma_F} H^1_{fl}(X_{F_v},A)[p^\infty] \longrightarrow
\prod_{v\in \Sigma_F} H^1_{fl}(X_{K_w},A)[p^\infty]^{G_v} \]
and \eqref{EqKerGamma} comes from the local version of the previous five term sequence (recalling that $H^2_{fl}(X_{F_v},A)=0$ by
\cite[Theorem III.7.8]{Mi1}).\\
If $v$ is unramified then, since we are assuming that $G$ has no element of order $p$, we have $G_v=0$ ($v$ is
totally split)
or $G_v\simeq \Z_p$ and $K_w=F_v^{unr}$ ($v$ is inert). When $G_v=0$ there is nothing to prove so we assume
$G_v\simeq \Z_p$ from now on.
The proof of \cite[Proposition I.3.8]{Mi1} can be generalized to show that
\[ H^i(G_v,A(K_w))\simeq H^i(G_v,\pi_0(\mathcal{A}(F_v)_0))\quad {\rm for\ any}\ i\geq 1\ ,\]
where $\mathcal{A}(F_v)_0$ is the closed fiber of the N\'eron model $\mathcal{A}(F_v)$ of $A$ at $v$ and $\pi_0$
denotes the set of connected components. Hence those groups are finite and trivial for $i\geqslant 2$, moreover,
if $v\not\in S$ (i.e., $v$ is of good reduction), then $H^1(G_v,A(K_w))=0$ as well.
\end{proof}

\subsection{Formula for the Euler characteristic}
We are now ready to prove our Euler characteristic formula for $Sel(K)\,$.

\begin{thm}\label{ThmEuChp}
With notations as above, assume that\begin{itemize}
\item[{\bf 1.}] $Sel(F)$ and $Sel_{A^t}(F)_p$ are finite (hence of order equal to $|\ts(A/F)[p^\infty]|$ and
$|\ts(A^t/F)[p^\infty]|$ respectively);
\item[{\bf 2.}] $\chi(G,A(K)[p^\infty])$ is well defined;
\item[{\bf 3.}] $H^i(G_v,A(K_w))[p^\infty]$ (for any $w$ lying over a ramified place $v$) is finite for $i=1,2$.
\end{itemize}
Then,
\[ \begin{array}{ll} \displaystyle{\frac{\chi(G,Sel(K))}{\chi^{(2)}(G,Sel(K))}} & =
\displaystyle{\frac{\chi(G,A(K)[p^\infty])}{\chi^{(3)}(G,A(K)[p^\infty])}
\cdot\frac{|\ts(A/F)[p^\infty]||NS(\psi_K)|}{|A^t(F)[p^\infty]||A(F)[p^\infty]|} }
\cdot \prod_{\begin{subarray}{c} v\in S \\v\ inert \end{subarray}}p^{\nu_v}\\
\ &  \cdot\displaystyle{\prod_{\begin{subarray}{c} v\in S \\v\ ram. \end{subarray}}
\frac{|H^1(G_v,A(K_w))[p^\infty]|}{|H^2(G_v,A(K_w))[p^\infty]|} }  \end{array} \]
where $p^{\nu_v}$ is the order of the $p$-part of the set of
connected components of the closed fiber of the N\'eron model of $A$ at $v$.
Moreover, for any ramified place $v$,  $|H^1(G_v,A(K_w))|=|\widehat{H}^0(G_v,A^t(K_w))|$
($\widehat{H}$ denotes the Tate cohomology groups).
\end{thm}

\begin{proof} We compute the terms of the formula of Lemma \ref{H^0/H^1}. By Lemma \ref{LocHochSer}
\[ \frac{|Coker(\beta)|}{|Ker(\beta)|}=\frac{|H^2(G,A(K)[p^\infty])|}{|H^1(G,A(K)[p^\infty])|} =
\frac{\chi(G,A(K)[p^\infty])}{\chi^{(3)}(G,A(K)[p^\infty])|A(F)[p^\infty]|} \ .\]
We are left with $Ker(\gamma)$ and $Coker(\gamma)$ and again use the description provided by Lemma \ref{LocHochSer}.
If $v$ ramifies in $K/F$, then, by \cite[Corollary 2.3.3]{Tan1}, $H^1(G_v,A(K_w))$ is the annihilator of the norm
from $K_w$ to $F_v$ of the group $A^t(K_w)$, with respect to the local Tate pairing. Hence
\[ \begin{array}{lcl} \displaystyle{ \prod_{\begin{subarray}{c} v\in S\\ v\ ram.\end{subarray}} |H^1(G_v,A(K_w))|} & = &
\displaystyle{ \prod_{\begin{subarray}{c} v\in S\\ v\ ram.\end{subarray}} |A^t(F_v)/N_{K_w/F_v}(A^t(K_w))| } \\
\ & = & \displaystyle{ \prod_{\begin{subarray}{c} v\in S\\ v\ ram.\end{subarray}} |\widehat{H}^0(G_v,A^t(K_w))|\ .} \end{array} \]
If $v\in S$ is unramified (hence of bad reduction for $A$), then, by \cite[Proposition I.3.8]{Mi1}
\[ H^1(G_v,A(K_w)) \simeq \left\{ \begin{array}{ll} 0 & {\rm if}\ G_v=0 \\
\ & \\
H^1(G_v,\pi_0(\mathcal{A}(F_v)_0)) & {\rm if}\ G_v\simeq \Z_p \end{array} \right. \ ,\]
(notations as in Lemma \ref{LocHochSer}). When $G_v\simeq \Z_p$ is procyclic (i.e., when $v$ is inert) we have
\[ |H^1(G_v,\pi_0(\mathcal{A}(F_v)_0))|=|H^0(G_v,\pi_0(\mathcal{A}(F_v)_0))|=|p{\rm -part\ of\ }\pi_0(\mathcal{A}(F_v)_0)| \]
and we denote this value by $p^{\nu_v}\,$. Therefore
\[ |Ker(\gamma)|= \prod_{\begin{subarray}{c} v\in S\\ v\ ram.\end{subarray}}  |\widehat{H}^0(G_v,A^t(K_w))[p^\infty]|\cdot
\prod_{\begin{subarray}{c} v\in S\\ v\ inert \end{subarray}} p^{\nu_v} \ .\]

For $Coker(\psi_F)$, fix a natural number $m$, then the commutative diagram
\[ \xymatrix{ A(F)/p^m A(F) \ar@{^(->}[r] \ar@{=}[d] & Sel_A(F,p^m) \ar@{->>}[r] \ar@{^(->}[d] &
\ts(A/F)[p^m] \ar@{^(->}[d] \\
A(F)/p^m A(F) \ar[d] \ar@{^(->}[r] & H^1_{fl}(X_F,A[p^m]) \ar@{->>}[r] \ar[d]^{\psi_{F,m}} &
H^1_{fl}(X_F,A)[p^m] \ar[d]^{\widetilde{\psi}_{F,m}} \\
 0 \ar[r] & \displaystyle{\prod_{v\in \Sigma} H^1_{fl}(X_{F_v},A[p^m])/Im(\kappa_v)} \ar[r]^{\quad\sim} &
\displaystyle{ \prod_{v\in \Sigma} H^1_{fl}(X_{F_v},A)[p^m]} } \]
shows that $Coker(\psi_{F,m})\simeq Coker(\widetilde{\psi}_{F,m})$.
Taking direct limits on $m$ and using \cite[Main Theorem]{GAV1}, one has
\[ Coker(\psi_F) \simeq Coker(\widetilde{\psi}_F) \simeq (T_p(Sel_{A^t}(F)_p))^\vee:= \left(\il{m} Sel_{A^t}(F,p^m)\right)^\vee \ .\]
Now assumption {\bf 1} yields $T_p(\ts(A^t/F)[p^m])=0$, hence the upper sequence of the diagram above
(substituting $A^t$ for $A$) shows that
\[ Coker(\psi_F) \simeq (T_p(Sel_{A^t}(F)_p))^\vee \simeq \left( \il{m} A^t(F)/p^m A^t(F)\right)^\vee:= (A^t(F)^*)^\vee \]
(the Pontrjagin dual of the $p$-adic completion $A^t(F)^*$ of $A^t(F)$, see also \cite[equation (6)]{GAV1}).
By hypothesis {\bf 1}, $A^t(F)$ is finite, hence
\[ A^t(F)^* \simeq A^t(F)[p^\infty] \ .\]
Now just substitute the computations above in the formula of Lemma \ref{H^0/H^1}.
\end{proof}

\begin{cor}\label{ThmEuChpZp}
Assume that $K/F$ is a $\Z_p$-extension (i.e., $G\simeq \Z_p$) and that $Sel(F)$ and $Sel_{A^t}(F)_p$ are finite.
Then,
\[ \chi(G,Sel(K)) = \frac{\chi(G,A(K)[p^\infty])|\ts(A/F)[p^\infty]||NS(\psi_K)|}{|A^t(F)[p^\infty]||A(F)[p^\infty]|}
\prod_{\begin{subarray}{c} v\in S \\v\ ram. \end{subarray}} |\widehat{H}^0(G_v,A^t(K_w))[p^\infty]|
\prod_{\begin{subarray}{c} v\in S \\v\ inert \end{subarray}} p^{\nu_v} \ . \]
%\[ \begin{array}{ll} \chi(G,Sel(K)) & = \displaystyle{\frac{\chi(G,A(K)[p^\infty])|\ts(A/F)[p^\infty]|}{|A^t(F)[p^\infty]||A(F)[p^\infty]|} }
%\cdot \displaystyle{\prod_{\begin{subarray}{c} v\in S \\v\ ram. \end{subarray}} } |\widehat{H}^0(G_v,A^t(K_w))[p^\infty]|\\
% &
%\displaystyle{\prod_{\begin{subarray}{c} v\in S \\v\ inert \end{subarray}} p^{\nu_v} }\ .
%\end{array} \]
\end{cor}

\begin{proof}
We have that $A(F)[p^\infty]$ is finite, then $H^1(G,A(K)[p^\infty])$
is finite as well (by \cite[Lemma 3.4]{BL2}) and  $H^i(G,A(K)[p^\infty])=0$ for any
$i\geq 2$: so $\chi(G,A(K)[p^\infty])$ is well defined. For the local terms let $v$ be a ramified place and
consider the sequence
\[ \xymatrix{ A(K_w)[p] \ar@{^(->}[r] & A(K_w) \ar@{->>}[r]^p & pA(K_w) } \ .\]
Taking $G_v$-cohomology one gets a surjection
\[ \xymatrix{ H^2(G_v,A(K_w)[p]) \ar@{->>}[r] & H^2(G_v,A(K_w))[p]  \ .}\]
The module on the left is trivial, so we have $H^2(G_v,A(K_w))[p]=0$ and $H^2(G_v,A(K_w))[p^\infty]=0$ as well.

\noindent Now just plug everything into (what remains of) the formula of the previous theorem.
\end{proof}

\begin{rem}\label{RemFinTor}
If $A(K)[p^\infty]$ is finite and $G\simeq \Z_p\,$, then
\[ |H^1(G,A(K)[p^\infty])| = |H^0(G,A(K)[p^\infty])| = |A(F)[p^\infty]| \]
and $\chi(G,A(K)[p^\infty])=1$. This is almost always the case (see, e.g., \cite[Proposition 2.11]{Tan2} and the
references mentioned there): for example it holds when $E$ is an elliptic curve by \cite[Theorem 4.2]{BLV}.
\end{rem}

\begin{cor}\label{CorEllCurp}
If $K=F^{ar}$ is the arithmetic $\Z_p$-extension of $F$ (i.e., generated by the $\Z_p$-extension of its field of constants),
$A=E$ is an elliptic curve and $Sel_E(F)_p$ is finite, then
\[ \chi(G,Sel_E(F^{ar})_p) =
 \frac{|\ts(E/F)[p^\infty]||NS(\psi_{F^{ar}})|}{|E(F)[p^\infty]|^2}
\prod_{\begin{subarray}{c} v\in S \\v\ inert \end{subarray}} p^{\nu_v} \ . \]
\end{cor}

\begin{proof} Just note that $E\simeq E^t$, $A(F^{ar})[p^\infty]$ is finite and there are no ramified primes in $F^{ar}/F$.
\end{proof}

\begin{rem}\label{RempsiKSurj}
When $\ell\neq p$ and $K=F(E[\ell^\infty])$ the map (analogous to $\psi_K$) which defines the $\l$-Selmer group
$Sel_E(K)_\l$ is surjective (see \cite[Theorem III.27]{S}). We are not aware of any other result on the surjectivity
of this kind of maps in positive characteristic: it would be interesting to investigate the subject further for general
global fields. Anyway, all the Euler characteristic formulas of this section would hold for surjective
$\psi_K$ just substituting $|NS(\psi_K)|$ with 1.
\end{rem}

\section{Akashi series for $\Z_p^d$-extensions}
The previous formulas for the Euler characteristic did not involve special values of $p$-adic $L$-functions (at least not directly).
The absence is mainly due to the lack of an appropriate definition of such functions in characteristic $p$: to fill this gap at least
in one case we compute here the Akashi series of the Pontrjagin dual of the $p$-Selmer group of a $\Z_p^d$-extension ($d<\infty$, but
the $d=\infty$ case for the totally ramified $\Z_p^\infty$-extensions described in \cite[Section 3]{BBL1} can be treated similarly
via a limit process). If $A$ is a constant abelian variety the $p$-adic $L$-function for this setting has been recently
provided in \cite{LLTT}.

\noindent Let $F_d$ be a field extension of $F$ such that $G=\Gal(F_d/F)\simeq \Z_p^d$ ($d\geqslant 2$) and fix topological generators
$\g_1,\dots,\g_d$ with $G_i=\overline{\langle \g_{i+1},\dots,\g_d\rangle}\simeq \Z_p^{d-i}$ ($0 \leq i \leq d-1$) and
$G/G_i\simeq \overline{\langle\g_1,\dots,\g_i\rangle} \simeq \Z_p^i$.
The picture of the field extension is the following:
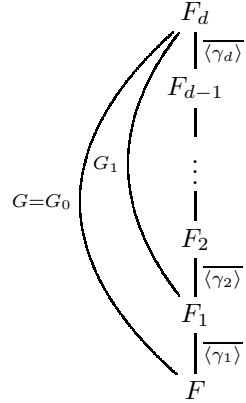
\[ \begin{xy}
(0,5)*+{F_d}="v1";(0,-5)*+{F_{d-1}}="v2";(0,-15)*+{\vdots}="v3";
(0,-25)*+{F_2}="v4";(0,-35)*+{F_1}="v5";(0,-45)*+{F}="v6";
{\ar@{-}^{\hbox{\scriptsize $\overline{\langle\gamma_d\rangle}$}} "v1";"v2"};
{\ar@{-} "v2";"v3"};{\ar@{-}"v3";"v4"};
{\ar@{-}^{\hbox{\scriptsize $\overline{\langle\gamma_2\rangle}$}} "v4";"v5"};
{\ar@{-}^{\hbox{\scriptsize $\overline{\langle\gamma_1\rangle}$}} "v5";"v6"};
{\ar@{-}@/_{3.6pc}/_{G=G_0} "v1";"v6"};
{\ar@{-}@/_{2.1pc}/_{G_1} "v1";"v5"};
\end{xy}  \]
For any group $H$ we let $\L(H)$ be the associated Iwasawa algebra and, to shorten notations, we set $\L_i:=\L(G/G_i)$
and let $\pi^i_j: \L_i \rightarrow \L_j$ be the natural projection for any $i>j$ (note that $G$ is abelian, hence
$\L_i\simeq\Z_p[[t_1,\dots,t_i]]$). For every $\L_i$-module $M$, we denote by $Ch_{\L_i}(M)=(\ch_{\L_i}(M))$ its characteristic ideal.
Moreover, to shorten notations, in this section we put $\Sel$ for the $p$-Selmer group of $F_d$, i.e., what we previously denoted
by $Sel(F_d)$ and $\S$ for its Pontrjagin dual $\Sel^\vee$. The fact that $\S\in \M_{G_1}(G)$ has been proved in many different
cases for example in \cite{BBL1}, \cite{BV1}, \cite{Tan2} and \cite{Wi}.

An important step here is the relation between $f_{\S,G_1}$ and $\ch_{\L_d}(\S)$: such a result can also be taken from
\cite[Lemmas 4.3 and 4.4]{CSS}, we give a different proof here (based on the following lemma) for completeness.

\begin{lem}\label{BBL-NY}
With notation as above let $s>1$, then, for every finitely generated torsion $\L_s$-module
$M$, we have:
\[  Ch_{\L_{s-1}}(M^{\overline{\langle\gamma_{s}\rangle}})\pi^{s}_{s-1}(Ch_{\L_s}(M))=Ch_{\L_{s-1}}(M/(\gamma_s-1)) \,.\]
\end{lem}

\begin{proof}
See \cite[Proposition 2.10]{BBL}.
\end{proof}

\begin{prop}\label{MainThmAk}
Assume $\S\in \M_{G_1}(G)$, then
%and let $H^j(G_i,\Sel)^\vee$ be finitely generated torsion $\L_i$-modules for any $1\leq i\leq d$ and any $j$. Then
\[ f_{\S,G_1} \sim \pi^d_1( \ch_{\L_d}(\S))\,.  \]
\end{prop}

\begin{proof}
Since $G_i/G_{i+1}\simeq \overline{\langle \gamma_{i+1}\rangle}$ has $p$-cohomological dimension $1$ we have the
following exact sequences
\[ H^1(\overline{\langle\gamma_{i+1}\rangle}, H^{j-1}(G_{i+1}, \Sel)) \iri H^j(G_i, \Sel) \sri
H^j(G_{i+1}, \Sel)^{\overline{\langle\gamma_{i+1}\rangle}}\quad \forall\ j\geqslant 1\,,\]
which yield (taking duals)
\begin{equation}\label{EqFin1}
H^j(G_{i+1}, \Sel)^\vee/(\gamma_{i+1}-1) \iri H^j(G_i, \Sel)^\vee \sri
(H^{j-1}(G_{i+1}, \Sel)^\vee)^{\overline{\langle\gamma_{i+1}\rangle}} \ \ 1\leqslant j\leqslant d-i-1
\end{equation}
and
\begin{equation}\label{EqFin2}
H^{d-i}(G_1, \Sel)^\vee \simeq (H^{d-i-1}(G_{i+1},\Sel)^\vee)^{\overline{\langle\gamma_{i+1}\rangle}}
\end{equation}
(because ${\rm cd}_p(G_{i+1})=d-i-1$).

We will use repeatedly these sequences and Lemma \ref{BBL-NY} with $M=H^j(G_i, \Sel)^\vee$.
Let us suppose that $d$ is odd (the argument for $d$ even is exactly the same), then:
\begin{align*}
 &  \displaystyle{\prod_{i=0}^{d-1}} (Ch_{\L_1}(H^i(G_1, \Sel)^\vee) )^{(-1)^i} \notag\\
 &= \frac{Ch_{\L_1}((\Sel^{G_1})^\vee)
\displaystyle{\prod_{\begin{subarray}{c} 1\leqslant i\leqslant d-4 \\ i\ odd \end{subarray}}}
Ch_{\L_1}((H^i(G_2,\Sel)^\vee)^{\overline{\langle\gamma_2\rangle}})
\displaystyle{\prod_{\begin{subarray}{c} 2\leqslant i\leqslant d-3 \\ i\ even\end{subarray}}}
Ch_{\L_1}(H^i(G_2, \Sel)^\vee/(\gamma_2 -1)) \cdot
Ch_{\L_1}(\S^{G_1})}{\displaystyle{\prod_{\begin{subarray}{c} 0\leqslant i\leqslant d-3 \\ i\ even \end{subarray}}}
Ch_{\L_1} ((H^i(G_2,\Sel)^\vee)^{\overline{\langle\gamma_2\rangle}})
\displaystyle{\prod_{\begin{subarray}{c} 1\leqslant i\leqslant d-2 \\ i\ odd\end{subarray}}}
Ch_{\L_1}(H^i(G_2, \Sel)^\vee/(\gamma_2 -1))} \notag\\
 &= \frac{Ch_{\L_1}((\Sel^{G_1})^\vee)
\displaystyle{\prod_{\begin{subarray}{c} 2\leqslant i\leqslant d-3 \\ i\ even\end{subarray}}}
\pi^2_1(Ch_{\L_2}(H^i(G_2,\Sel)^\vee)) \cdot
Ch_{\L_1}(\S^{G_1})}{Ch_{\L_1}((H^0(G_2,\Sel)^\vee)^{\overline{\langle \gamma_2\rangle}})
\displaystyle{\prod_{\begin{subarray}{c} 1\leqslant i\leqslant d-4 \\ i\ odd\end{subarray}}}
\pi^2_1(Ch_{\L_2}(H^1(G_2,\Sel)^\vee))\ \cdot Ch_{\L_1}(H^{d-2}(G_2,\Sel)^\vee/(\gamma_2-1))}\,.
\end{align*}

Observe that (by equation \eqref{EqFin2})
\begin{align*}
Ch_{\L_1}(H^{d-2}(G_2, \Sel)^\vee/(\gamma_2 -1)) & =
Ch_{\L_1} \left((H^{d-2}(G_2, \Sel)^\vee)^{\overline{\langle\gamma_2\rangle}}\right) \pi^2_1\left( Ch_{\L_2} (\S^{G_2}) \right) \notag\\
\ & = Ch_{\L_1} \left(\left(\left(
H^{d-3}(G_3, \Sel)^\vee \right)^{\overline{\langle\gamma_3\rangle}}\right)^{\overline{\langle\gamma_2\rangle}}\right)
\pi^2_1\left( Ch_{\L_2} (\S^{G_2}) \right) \\
\ & \ \  \vdots  \\
%&= Ch_{\L_1} \left(\left(\left( \left( \S^{\overline{\langle\gamma_d\rangle}}
%\right)^{\iddots} \right)^{\overline{\langle\gamma_3\rangle}}\right)^{\overline{\langle\gamma_2\rangle}}\right)
%\pi^2_1\left( Ch_{\L_2} (\S^{G_2}) \right) \notag \\
\ & = Ch_{\L_1}(\S^{G_1}) \pi^2_1\left( Ch_{\L_2} (\S^{G_2}) \right)
\end{align*}
and (by Lemma \ref{BBL-NY})
\begin{align*}
Ch_{\L_1}((H^0(G_2,\Sel)^\vee)^{\overline{\langle \gamma_2\rangle}}) & = \displaystyle{
\frac{Ch_{\L_1}(H^0(G_2,\Sel)^\vee/(\gamma_2-1))}{\pi^2_1(Ch_{\L_2}(H^0(G_2,\Sel)^\vee))} } \notag\\
\ &  =\displaystyle{ \frac{Ch_{\L_1}((\Sel^{G_1})^\vee)}{\pi^2_1(Ch_{\L_2}(\Sel^{G_2})^\vee)} }\,.
\end{align*}

Then

\begin{align*}
 &  \displaystyle{\prod_{i=0}^{d-1}} (Ch_{\L_1}(H^i(G_1, \Sel)^\vee) )^{(-1)^i} \notag\\
& = \frac{\pi^2_1(Ch_{\L_2}(\Sel^{G_2})^\vee)
\displaystyle{\prod_{\begin{subarray}{c} 2\leqslant i\leqslant d-3 \\ i\ even\end{subarray}}}\pi^2_1(Ch_{\L_2}(H^i(G_2,\Sel)^\vee))}
{\displaystyle{\prod_{\begin{subarray}{c} 1\leqslant i\leqslant d-4 \\ i\ odd\end{subarray}}} \pi^2_1(Ch_{\L_2}(H^i(G_2,\Sel)^\vee))\cdot
\pi^2_1\left( Ch_{\L_2} (\S^{G_2}) \right)} \notag \\
 &= \frac{\pi^2_1(Ch_{\L_2}(\Sel^{G_2})^\vee)
\displaystyle{\prod_{\begin{subarray}{c} 1\leqslant i\leqslant d-4 \\ i\ odd\end{subarray}}}
\pi^2_1(Ch_{\L_2}(H^i(G_3,\Sel)^\vee)^{\overline{\langle \gamma_3\rangle}})
\displaystyle{\prod_{\begin{subarray}{c} 2\leqslant i\leqslant d-3 \\ i\ even\end{subarray}}}
\pi^2_1(Ch_{\L_2}(H^i(G_3,\Sel)^\vee/(\gamma_3-1)))}
{\displaystyle{\prod_{\begin{subarray}{c} 0\leqslant i\leqslant d-5 \\ i\ even\end{subarray}}}
\pi^2_1(Ch_{\L_2}(H^i(G_3,\Sel)^\vee)^{\overline{\langle \gamma_3\rangle}})
\displaystyle{\prod_{\begin{subarray}{c} 1\leqslant i\leqslant d-4 \\ i\ odd\end{subarray}}}
\pi^2_1(Ch_{\L_2}(H^i(G_3,\Sel)^\vee/(\gamma_3-1)))\cdot \pi^2_1\left( Ch_{\L_2} (\S^{G_2}) \right)} \notag \\
 &= \frac{\pi^3_1(Ch_{\L_3}(\Sel^{G_3})^\vee)
\displaystyle{\prod_{\begin{subarray}{c} 2\leqslant i\leqslant d-5 \\ i\ even\end{subarray}}}
\pi^3_1(Ch_{\L_3}(H^i(G_3,\Sel)^\vee))\cdot \pi^3_1(Ch_{\L_3}(\S^{G_3}))}
{\displaystyle{\prod_{\begin{subarray}{c} 1\leqslant i\leqslant d-4 \\ i\ odd\end{subarray}}}
\pi^3_1(Ch_{\L_3}(H^1(G_3, \Sel)^\vee))} \\
 & \ \ \vdots \\
 & = \pi^d_1(Ch_{\L_d}(\S)) \,.
\end{align*}
Now since
\[ f_{\S,G_1}:= \prod_{i=0}^{d-1} \ch_{\L_1}(H_i(G_1, \S))^{(-1)^i} = \prod_{i=0}^{d-1} \ch_{\L_1}(H^i(G_1, \Sel)^\vee)^{(-1)^i} \,,\]
\[ (\ch_{\L_1}(H^i(G_1, \Sel)^\vee)) = Ch_{\L_1}(H^i(G_1, \Sel)^\vee) \quad{\rm and}\quad
(\ch_{\L_d}(\S))= Ch_{\L_d}(\S)\]
the proposition follows.
\end{proof}

The following proposition (see \cite[Lemma 4.2]{CSS}) gives us a relation between the Euler characteristic
of $\S$ and its Akashi series.

\begin{prop}\label{WhyEuCh}
Assume that $\S$ has finite Euler characteristic, then the Akashi series of $\S$ is well defined and non-zero, and we have
\[  \chi(G,\S) = |\pi^d_0(f_{\S,G_1})|^{-1}_p \]
where $|\cdot|_p$ stands for the usual $p$-adic absolute value.
\end{prop}

\subsection{Application to constant ordinary abelian varieties}
This final section briefly presents the main results of \cite{LLTT} (for all the details see the original paper
and its companion \cite{LLTT1}), in order to provide a simple but (in our opinion) meaningful application
of our formula to the setting of constant ordinary abelian varieties. We keep notations as close as possible to
the ones in \cite{LLTT}.

Let $A$ be a constant ordinary abelian variety defined over the constant field $\F$ of $F$, then $\Gal(\overline{\F}/\F)$ acts
on $A[p^\infty]\simeq (\Q_p/\Z_p)^g$ via a {\em twist matrix} {\bf u} whose eigenvalues we denote by $\alpha_1,\dots,\alpha_g$
(counted with multiplicities). Let $\mathcal{O}$ be the ring of integers of a finite extension of $\Q_p$ containing
all the $\alpha_i$ and note that, in particular, $\alpha_i\in\mathcal{O}^*$ for any $i$. Assuming $S\neq\emptyset$ one can define
a Stickelberger series
\[ \Theta_{F_d,S}(u) := \prod_{v\not\in S}(1-[v]_{F_d}u^{\deg(v)})^{-1} \in \Lambda_d[[u]] \,,\]
where $[v]_{F_d}\in \Gal(F_d/F)$ is the arithmetic Frobenius at $v$ (the case $S=\emptyset$ just needs an extra factor
$1-Fr_q\cdot u$). It is not hard to see that $\Theta_{F_d,S}(u)$
behaves well under projections, i.e., $\pi^d_{d-1}(\Theta_{F_d,S}(u))=\Theta_{F_{d-1},S}(u)$.

\begin{defin}(The $p$-adic $\mathcal{L}$-function) Let
\[ \theta^+_{A,F_d,S} := \prod_{i=1}^g \Theta_{F_d,S}(\alpha_i^{-1})^\# \in \mathcal{O}[[G]] \]
be the {\em Stickelberger element} (where $\cdot^\#$ denotes the inversion $\mathcal{O}[[G]]\rightarrow\mathcal{O}[[G]]$,
$\gamma \rightarrow \gamma^{-1}$ for any $\gamma \in G$, all issues about this being a good definition and convergence are dealt
with in \cite{LLTT}) and define
\[ \mathcal{L}_{A,F_d} := \theta^+_{A,F_d,S}\cdot (\theta^+_{A,F_d,S})^\# \in \mathcal{O}[[G]] \]
as the {\em $p$-adic $\mathcal{L}$-function} associated to $A$ and $F_d$.
\end{defin}

Using a deep relation between duals of Selmer groups and divisor class groups (which Stickelberger elements are
usually associated to, see \cite[Proposition 3.18]{LLTT}), one can prove an interpolation formula ({\bf IF}) and
an Iwasawa Main Conjecture ({\bf IMC}) for $\mathcal{L}_{A,F_d}$.

\begin{thm} \label{LLTTThm}{\em (\cite[Theorems 4.7 and 4.9]{LLTT})}
For any continuous character $\omega: G\rightarrow \mathbb{C}^*$ one has
\[ {\rm ({\bf IF})}\qquad \omega(\mathcal{L}_{A,F_d})=c_{A,F_d,\omega}\cdot L(A,\omega,1) \,,\]
where $c_{A,F_d,\omega}$ is an explicit fudge factor and $L$ is the classic $L$-function of $A$ twisted by $\omega$.
Moreover
\[ {\rm ({\bf IMC})}\qquad \qquad Ch_{\mathcal{O}[[G]]}(\S)=(\mathcal{L}_{A,F_d}) \]
as ideals of $\mathcal{O}[[G]]$.
\end{thm}

This deep result allows us to conclude with the following

\begin{cor}\label{AkSerStickEl}
With notations as above, one has
\[ \chi(G,\S) = |\pi^d_0(\mathcal{L}_{A,F_d})|^{-1}_p \,.\]
\end{cor}

\begin{proof}
One simply needs to apply Propositions \ref{MainThmAk} and \ref{WhyEuCh} to the {\bf IMC} formula provided by
Theorem \ref{LLTTThm}.
\end{proof}

\end{document}